\documentclass[11pt,a4paper]{amsart}
\usepackage{amsfonts,amsmath,amssymb,amsthm,mathtools}
\usepackage{times}
\usepackage{color}
\usepackage{stmaryrd}

\numberwithin{equation}{section}

\DeclareSymbolFont{SY}{U}{psy}{m}{n}
\DeclareMathSymbol{\emptyset}{\mathord}{SY}{'306}

\DeclareMathOperator{\Ran}{Ran}
\DeclareMathOperator{\Dom}{Dom}
\DeclareMathOperator{\spec}{spec}
\DeclareMathOperator{\dist}{dist}
\DeclareMathOperator{\conv}{conv}

\renewcommand{\Re}{\operatorname{Re}}

\DeclarePairedDelimiter{\abs}{\lvert}{\rvert}
\DeclarePairedDelimiter{\norm}{\lVert}{\rVert}
\DeclarePairedDelimiter{\tnorm}{\interleave}{\interleave}

\newcommand{\dd}{\mathrm d}
\newcommand{\ee}{\mathrm e}
\newcommand{\ii}{\mathrm i}

\newcommand{\N}{\mathbb{N}}
\newcommand{\R}{\mathbb{R}}
\newcommand{\EE}{\mathsf{E}}

\newcommand{\cH}{{\mathcal H}}
\newcommand{\cK}{{\mathcal K}}
\newcommand{\cL}{{\mathcal L}}
\newcommand{\cO}{{\mathcal O}}
\newcommand{\cU}{{\mathcal U}}

\newcommand{\fS}{\mathfrak{S}}

\newtheorem{introtheorem}{Theorem}{\bf}{\it}
\newtheorem{introcorollary}[introtheorem]{Corollary}{\bf}{\it}
\newtheorem{theorem}{Theorem}[section]{\bf}{\it}
\newtheorem{lemma}[theorem]{Lemma}{\bf}{\it}
\newtheorem{proposition}[theorem]{Proposition}{\bf}{\it}
\newtheorem{remark}[theorem]{Remark}{\it}{\rm}
{\bf}{\it}

\title[Notes on the $\sin 2\Theta$ theorem]{Notes on the $\sin 2\Theta$ theorem}
\date{}

\author[A.\ Seelmann]{Albrecht Seelmann$^*$}
\address{A.~Seelmann, FB 08 - Institut f\"{u}r Mathematik,
Johannes Gutenberg-Universi\-t\"{a}t Mainz,
Staudinger Weg 9,
D-55099 Mainz,
Germany}
\email{seelmann@mathematik.uni-mainz.de}
\thanks{$^*$The material presented in this work will be part of the author's Ph.D. thesis.}

\begin{document}

%%%%%%%%%%%%%%%%%%%%%%%%%%%%%%%%%%%%%%%%%%%%%%%%%%%%%%%%%%%%%%%%%%%%%%%%%%%%%%%%%%%%%%%%%%%%%%%%%%%%%%%%%%%%%%%%%%%%%%%%%%%%%%%%%%%
%%% Abstract
%%%%%%%%%%%%%%%%%%%%%%%%%%%%%%%%%%%%%%%%%%%%%%%%%%%%%%%%%%%%%%%%%%%%%%%%%%%%%%%%%%%%%%%%%%%%%%%%%%%%%%%%%%%%%%%%%%%%%%%%%%%%%%%%%%%
\begin{abstract}
An analogue of the Davis-Kahan $\sin2\Theta$ theorem from [SIAM J.\ Numer.\ Anal.\ \textbf{7} (1970), 1--46] is proved under a
general spectral separation condition. This extends the generic $\sin2\theta$ estimates recently shown by Albeverio and Motovilov
in [Complex Anal.\ Oper.\ Theory \textbf{7} (2013), 1389--1416]. The result is applied to the subspace perturbation problem to
obtain a bound on the arcsine of the norm of the difference of the spectral projections associated with isolated components of the
spectrum of the unperturbed and perturbed operators, respectively.
\end{abstract}

\maketitle

%%%%%%%%%%%%%%%%%%%%%%%%%%%%%%%%%%%%%%%%%%%%%%%%%%%%%%%%%%%%%%%%%%%%%%%%%%%%%%%%%%%%%%%%%%%%%%%%%%%%%%%%%%%%%%%%%%%%%%%%%%%%%%%%%%%
%%%%%%%%%%%%%%%%%%%%%%%%%%%%%%%%%%%%%%%%%%%%%%%%%%%%%%%%%%%%%%%%%%%%%%%%%%%%%%%%%%%%%%%%%%%%%%%%%%%%%%%%%%%%%%%%%%%%%%%%%%%%%%%%%%%
%%% Introduction
%%%%%%%%%%%%%%%%%%%%%%%%%%%%%%%%%%%%%%%%%%%%%%%%%%%%%%%%%%%%%%%%%%%%%%%%%%%%%%%%%%%%%%%%%%%%%%%%%%%%%%%%%%%%%%%%%%%%%%%%%%%%%%%%%%%
%%%%%%%%%%%%%%%%%%%%%%%%%%%%%%%%%%%%%%%%%%%%%%%%%%%%%%%%%%%%%%%%%%%%%%%%%%%%%%%%%%%%%%%%%%%%%%%%%%%%%%%%%%%%%%%%%%%%%%%%%%%%%%%%%%%
\section{Introduction and main results}
The \emph{subspace perturbation problem} is a fundamental problem in operator perturbation theory that deals with the variation of
spectral subspaces for a self-adjoint or normal operator under a perturbation, see, e.g., \cite{BDM83} and \cite{KMM03}. Some of
the major contributions to this field of research have been made by Davis and Kahan \cite{DK70}. Extensions and generalizations of
their results have been considered in several works such as \cite{AM12,AM12:2,KMM04,KMM05,KMM07,MoSe06}.

The main objective in these studies is to bound trigonometric functions of the operator angle associated with spectral subspaces of
the unperturbed and perturbed operators, respectively.

Our main result is the following variant of the Davis-Kahan $\sin2\Theta$ theorem.
\begin{introtheorem}\label{thm:intro}
 Let $A$ be a self-adjoint operator on a separable Hilbert space $\cH$ such that the spectrum of $A$ is separated into two disjoint
 components, that is,
 \[
  \spec(A) = \sigma \cup \Sigma \quad\text{ with }\quad d:=\dist(\sigma,\Sigma)>0\,.
 \]
 Moreover, let $V$ be a bounded self-adjoint operator on $\cH$, and let $Q$ be an orthogonal projection in $\cH$ onto a reducing
 subspace for $A+V$. Then, the operator angle $\Theta=\Theta(\EE_A(\sigma),Q)$ associated with the subspaces $\Ran\EE_A(\sigma)$
 and $\Ran Q$ satisfies
 \begin{equation}\label{eq:introSin2Theta}
  \norm{\sin2\Theta} \le \frac{\pi}{2}\cdot 2\,\frac{\norm{V}}{d}\,.
 \end{equation}
 Here, $\EE_A(\sigma)$ denotes the spectral projection for $A$ associated with $\sigma$.
\end{introtheorem}

For a definition of the operator angle associated with two closed subspaces of a Hilbert space, a self-adjoint operator whose
spectrum lies in the interval $[0,\frac{\pi}{2}]$, see Section \ref{sec:sin2Theta} below; see also \cite{KMM03:2}.

It should be emphasized that the projection $Q$ in Theorem \ref{thm:intro} is not assumed to be a spectral projection for $A+V$ and
that the operator $A$ is allowed to be unbounded. It is also worth mentioning that the bound \eqref{eq:introSin2Theta} is of an
\emph{a priori} type since the spectral separation condition is imposed on the unperturbed operator $A$ only. By switching the
roles of $A$ and $A+V$, one may impose the analogous condition on the perturbed operator $A+V$ instead, which results in the
corresponding \emph{a posteriori} type estimate.

Theorem \ref{thm:intro} is a direct analogue of the Davis-Kahan $\sin2\Theta$ theorem from \cite{DK70}. There, it is additionally
assumed that the convex hull of one of the spectral components $\sigma$ and $\Sigma$ is disjoint from the other one, that is,
$\conv(\sigma)\cap\Sigma=\emptyset$ or vice versa. The corresponding estimate is the same as \eqref{eq:introSin2Theta}, except for
the constant $\frac{\pi}{2}$ being replaced by $1$. Note that the Davis-Kahan $\sin2\Theta$ theorem is formulated in \cite{DK70}
for arbitrary unitary-invariant norms including the standard Schatten norms. A corresponding extension of Theorem \ref{thm:intro}
is discussed in Section \ref{sec:snIdeals} below.

An immediate consequence of Theorem \ref{thm:intro} is the \emph{generic $\sin2\theta$ estimate} recently proved by Albeverio and
Motovilov in \cite{AM12:2},
\begin{equation}\label{eq:introSin2theta}
 \sin2\theta \le \frac{\pi}{2}\cdot 2\,\frac{\norm{V}}{d}\quad\text{ with }\quad
 \theta:=\norm{\Theta}=\arcsin\bigl(\norm{\EE_A(\sigma)-Q}\bigr)\,.
\end{equation}
This is due to the elementary inequality $\sin(2\norm{\Theta})\le\norm{\sin2\Theta}$; for the second representation of $\theta$ in
\eqref{eq:introSin2theta}, see equation \eqref{eq:projDiffNorm} below. In this respect, we may call \eqref{eq:introSin2Theta} the
\emph{generic $\sin2\Theta$ estimate}. It should be emphasized that, in contrast to \eqref{eq:introSin2Theta}, no extension of
\eqref{eq:introSin2theta} to norms other than the usual bound norm is at hand.

Clearly, the estimates \eqref{eq:introSin2Theta} and \eqref{eq:introSin2theta} provide no useful information if
$\norm{V}\ge\frac{d}{\pi}$. On the other hand, for perturbations $V$ satisfying $\norm{V}<\frac{d}{\pi}$, the $\sin2\Theta$
estimate \eqref{eq:introSin2Theta} implies that $\norm{\sin2\Theta}<1$, so that the spectrum of $\Theta$ has a gap around
$\frac{\pi}{4}$. This means that there is an open interval containing $\frac{\pi}{4}$ that belongs to the resolvent set of
$\Theta$, namely
\[
 \Bigl(\alpha, \frac{\pi}{2}-\alpha\Bigr) \subset\Bigl[0,\frac{\pi}{2}\Bigr]\setminus\spec(\Theta)\quad\text{ with }\quad
 \alpha:=\frac{1}{2}\arcsin\Bigl(\frac{\pi}{2}\cdot2\,\frac{\norm{V}}{d}\Bigr)<\frac{\pi}{4}\,.
\]
Note that $\Theta$ may a priori have spectrum both in $[0,\alpha]$ and $\bigl[\frac{\pi}{2}-\alpha,\frac{\pi}{2}\bigr]$. This
depends on the reducing subspace for $A+V$ that is considered, see Remark \ref{rem:opAngle} below.

In this regard, Theorem \ref{thm:intro} is in general stronger than the corresponding result of the $\sin2\theta$ estimate
\eqref{eq:introSin2theta} since the latter provides information only on the \emph{maximal angle} $\theta=\norm{\Theta}$ between the
subspaces $\Ran\EE_A(\sigma)$ and $\Ran Q$, cf.\ \cite[Remark 4.2]{AM12:2}. However, if it is known that $\theta\le\frac{\pi}{4}$,
then $\norm{\sin2\Theta}=\sin2\theta$, so that, in this case, both estimates agree.

As an application to the subspace perturbation problem, we obtain the following bound on the maximal angle between the
corresponding spectral subspaces for the unperturbed and perturbed operators $A$ and $A+V$, respectively.
\begin{introcorollary}[cf.\ {\cite[Remark 4.4]{AM12:2}}]\label{cor:intro}
 Let $A$ and $V$ be as in Theorem \ref{thm:intro}. If $\norm{V}\le \frac{d}{\pi}$, then
 \begin{equation}\label{eq:introMaxAngleBound}
  \arcsin\bigl(\norm{\EE_A(\sigma) - \EE_{A+V}\bigl(\cO_{d/2}(\sigma)\bigr)}\bigr) \le
  \frac{1}{2}\arcsin\Bigl(\frac{\pi}{2}\cdot 2\,\frac{\norm{V}}{d}\Bigr) \le \frac{\pi}{4}\,,
 \end{equation}
 where $\EE_{A+V}\bigl(\cO_{d/2}(\sigma)\bigr)$ denotes the spectral projection for $A+V$ associated with the open
 $\frac{d}{2}$-neighbourhood $\cO_{d/2}(\sigma)$ of $\sigma$.
\end{introcorollary}

The bound \eqref{eq:introMaxAngleBound} in Corollary \ref{cor:intro} is not optimal if $\norm{V}>\frac{4}{\pi^2+4}\,d$, see
\cite{AM12:2}. However, for perturbations $V$ satisfying $\norm{V}\le\frac{4}{\pi^2+4}\,d$, this bound on the maximal angle is the
strongest one available so far, cf.\ \cite[Remark 5.5]{AM12:2}.

The paper is organized as follows:
Section \ref{sec:sin2Theta} forms the main part of this work. There, we recall the notion of the operator angle associated with
two closed subspaces of a Hilbert space and give the proofs of Theorem \ref{thm:intro} and Corollary \ref{cor:intro}. We also
discuss the statement of Corollary \ref{cor:intro} in the situation of the original Davis-Kahan $\sin2\Theta$ theorem, see Remark
\ref{rem:sin2Theta}.

Section \ref{sec:sin2theta} is devoted to an alternative, straightforward proof of the $\sin2\theta$ estimate
\eqref{eq:introSin2theta} that is not based on Theorem \ref{thm:intro} and is more direct than the one by Albeverio and Motovilov
in \cite{AM12:2}.

Finally, in Section \ref{sec:snIdeals} we extend Theorem \ref{thm:intro} to symmetrically-normed ideals such as the standard
Schatten classes, see Theorem \ref{thm:sin2ThetasnIdeals}. We also discuss the more general case of normal operators $A$ in this
section, see Remark \ref{rem:sin2ThetasnIdeals}.

Throughout this paper we use the following notations:

The space of bounded linear operators from a Hilbert space $\cH$ to a Hilbert space $\cK$ is denoted by $\cL(\cH,\cK)$, and
$\norm{\cdot}$ stands for the usual bound norm on $\cL(\cH,\cK)$. If $\cH=\cK$, we simply write $\cL(\cH):=\cL(\cH,\cH)$. The
identity operator on $\cH$ is denoted by $I_{\cH}$. For an orthogonal projection $P$ in $\cH$ we write $P^\perp:=I_{\cH}-P$.

Given a linear operator $B$ on a Hilbert space, its domain and its range are written as $\Dom(B)$ and $\Ran(B)$, respectively.
Finally, if $B$ is self-adjoint, $\EE_B(\Delta)$ denotes the spectral projection for $B$ associated with a Borel set
$\Delta\subset\R$.

%%%%%%%%%%%%%%%%%%%%%%%%%%%%%%%%%%%%%%%%%%%%%%%%%%%%%%%%%%%%%%%%%%%%%%%%%%%%%%%%%%%%%%%%%%%%%%%%%%%%%%%%%%%%%%%%%%%%%%%%%%%%%%%%%%%
%%%%%%%%%%%%%%%%%%%%%%%%%%%%%%%%%%%%%%%%%%%%%%%%%%%%%%%%%%%%%%%%%%%%%%%%%%%%%%%%%%%%%%%%%%%%%%%%%%%%%%%%%%%%%%%%%%%%%%%%%%%%%%%%%%%
%%% Proofs
%%%%%%%%%%%%%%%%%%%%%%%%%%%%%%%%%%%%%%%%%%%%%%%%%%%%%%%%%%%%%%%%%%%%%%%%%%%%%%%%%%%%%%%%%%%%%%%%%%%%%%%%%%%%%%%%%%%%%%%%%%%%%%%%%%%
%%%%%%%%%%%%%%%%%%%%%%%%%%%%%%%%%%%%%%%%%%%%%%%%%%%%%%%%%%%%%%%%%%%%%%%%%%%%%%%%%%%%%%%%%%%%%%%%%%%%%%%%%%%%%%%%%%%%%%%%%%%%%%%%%%%
\section{Proof of Theorem \ref{thm:intro} and Corollary \ref{cor:intro}}\label{sec:sin2Theta}

%%%%%%%%%%%%%%%%%%%%%%%%%%%%%%%%%%%%%%%%%%%%%%%%%%%%%%%%%%%%%%%%%%%%%%%%%%%%%%%%%%%%%%%%%%%%%%%%%%%%%%%%%%%%%%%%%%%%%%%%%%%%%%%%%%%
%%% Separation of subspaces
%%%%%%%%%%%%%%%%%%%%%%%%%%%%%%%%%%%%%%%%%%%%%%%%%%%%%%%%%%%%%%%%%%%%%%%%%%%%%%%%%%%%%%%%%%%%%%%%%%%%%%%%%%%%%%%%%%%%%%%%%%%%%%%%%%%
We start with briefly recalling some well-known facts on the separation of closed subspaces. For a more detailed discussion on this
material, see \cite{Davis58} and \cite[Section 3]{DK70}; see also \cite{AM12:2}, \cite{KMM03:2}, \cite{MoSe06}, and references
therein.

Let $P$ and $Q$ be two orthogonal projections in a Hilbert space $\cH$. Following \cite{Davis58}, we introduce the \emph{closeness
operator}
\[
 C := C(P,Q) := PQP + P^\perp Q^\perp P^\perp
\]
and the \emph{separation operator}
\[
 S := S(P,Q) := PQ^\perp P + P^\perp QP^\perp\,.
\]
Since $P$ and $Q$ are self-adjoint, $C$ and $S$ are self-adjoint as well. Moreover, one has
\begin{equation}\label{eq:CS}
 0 \le C \le I_{\cH}\,,\quad  0 \le S \le I_{\cH}\,, \quad\text{ and }\quad C + S = I_{\cH}\,.
\end{equation}

The \emph{operator angle} associated with the subspaces $\Ran P$ and $\Ran Q$ can now be introduced via the functional calculus as
\begin{equation}\label{eq:defOpAngle}
 \Theta := \Theta(P,Q) := \arccos\bigl( \sqrt{C(P,Q)}\, \bigr)\,.
\end{equation}
Clearly, $\Theta$ is self-adjoint and its spectrum lies in the interval $\bigl[0,\frac{\pi}{2}\bigr]$. Furthermore, taking into
account \eqref{eq:CS} and \eqref{eq:defOpAngle}, the operators $C$ and $S$ can be represented as
\begin{equation}\label{eq:CSTheta}
 C=\cos^2\Theta \quad\text{ and }\quad S=\sin^2\Theta\,.
\end{equation}

It should be mentioned that in many works such as \cite{KMM03:2} and \cite{KMM05} the operator angle is introduced in a slightly
different way. There, instead of $\Theta$ in \eqref{eq:defOpAngle}, its restriction to $\Ran P$, or even to the maximal subspace of
$\Ran P$ where it has trivial kernel, is considered. The above definition \eqref{eq:defOpAngle} follows the approach by Davis and
Kahan (cf.\ \cite[Eqs.\ (1.16) and (1.17)]{DK70}; see also \cite[p.\ 17]{DK70}) and provides a generalization of their notion of
the operator angle, see Remark \ref{rem:opAngle} below.

As in \cite[Section 34]{AG93}, one has
\begin{equation}\label{eq:projDiff}
 P-Q=P(I_{\cH}-Q)-(I_{\cH}-P)Q = PQ^\perp - P^\perp Q = Q^\perp P - QP^\perp\,,
\end{equation}
so that
\[
 \begin{aligned}
  (P-Q)^2 &= \bigl(PQ^\perp - P^\perp Q\bigr)\bigl(Q^\perp P - QP^\perp\bigr)\\
  &= PQ^\perp P + P^\perp QP^\perp = S = \sin^2\Theta\,,
 \end{aligned}
\]
that is,
\begin{equation}\label{eq:sinTheta}
 \abs{P-Q} = \sin\Theta\,.
\end{equation}
In particular,
\begin{equation}\label{eq:projDiffNorm}
 \norm{P-Q}=\norm{\sin\Theta}=\sin\norm{\Theta}\le1\,.
\end{equation}
Thus, suitable norms of the operator angle $\Theta$ or of trigonometric functions thereof can be used to measure the difference
between the subspaces $\Ran P$ and $\Ran Q$.

\begin{remark}\label{rem:dirRot}
 If the subspaces $\Ran P$ and $\Ran Q$ are \emph{equivalently positioned} (cf.\ \cite{Davis58}), that is,
 \[
  \dim \Ran P\cap \Ran Q^\perp = \dim \Ran P^\perp \cap \Ran Q\,,
 \]
 then there exists a unitary operator $U\in\cL(\cH)$ such that
 \[
  QU=UP\,,\quad U^2 = (Q-Q^\perp)(P-P^\perp)\,, \quad\text{ and }\quad \Re U\ge 0\,,
 \]
 where $\Re U=\frac{1}{2}(U+U^*)$ denotes the real part of $U$, see, e.g., \cite[Propositions 3.2 and 3.3]{DK70}; such a unitary
 operator is called a \emph{direct rotation} from $\Ran P$ to $\Ran Q$. In this case, it is an elementary exercise to check that
 \[
  \Re U = \sqrt{C} = \cos\Theta\,,
 \]
 so that the operator angle $\Theta$ has indeed a natural interpretation as a rotation angle. Following \cite[Eq.\ (1.18)]{DK70},
 the operator $U$ can even be represented as a $2\times 2$ block operator matrix that resembles a rotation matrix from the
 $2$-dimensional case.
 
 However, the definition \eqref{eq:defOpAngle} of the operator angle does not require that a direct rotation from $\Ran P$ to
 $\Ran Q$ exists. In fact, it is not even necessary that a unitary operator taking $\Ran P$ to $\Ran Q$ exists at all. In this
 respect, \eqref{eq:defOpAngle} generalizes the notion of the operator angle from \cite{DK70}, and it turns out to be very
 convenient to formulate our considerations in this generality.
\end{remark}

%%%%%%%%%%%%%%%%%%%%%%%%%%%%%%%%%%%%%%%%%%%%%%%%%%%%%%%%%%%%%%%%%%%%%%%%%%%%%%%%%%%%%%%%%%%%%%%%%%%%%%%%%%%%%%%%%%%%%%%%%%%%%%%%%%%
%%% Davis-Kahan: \sin2\Theta to \sin\Theta
%%%%%%%%%%%%%%%%%%%%%%%%%%%%%%%%%%%%%%%%%%%%%%%%%%%%%%%%%%%%%%%%%%%%%%%%%%%%%%%%%%%%%%%%%%%%%%%%%%%%%%%%%%%%%%%%%%%%%%%%%%%%%%%%%%%
The following result has already played a crucial role in the proof of the original Davis-Kahan $\sin2\Theta$ theorem in
\cite{DK70}, and it is one of the key ingredients for our proof of Theorem \ref{thm:intro} as well.
\begin{lemma}[cf.\ {\cite[Section 7]{DK70}}]\label{lem:sin2ThetasinTheta}
 Let $P$ and $Q$ be two orthogonal projections in a Hilbert space $\cH$, and denote $K := Q-Q^\perp$. Then
 \[
  \sin\bigl(2\Theta(P,Q)\bigr) = \sin\bigl(\Theta(P,KPK)\bigr)\,.
 \]
 \begin{proof}
  For the sake of completeness, we give a proof in the current notations.

  In view of \eqref{eq:CSTheta}, one computes
  \begin{equation}\label{eq:sin2ThetaSqr}
   \begin{aligned}
    \sin^2\bigl(2\Theta(P,Q)\bigr)
    &= 4S(P,Q)C(P,Q)\\
    &= 4 \bigl(PQ^\perp P + P^\perp QP^\perp\bigr) \bigl(PQP + P^\perp Q^\perp P^\perp\bigr)\\
    &= 4 PQ^\perp PQP + 4 P^\perp Q P^\perp Q^\perp P^\perp \,.
   \end{aligned}
  \end{equation}
 Denote $R:=KPK$. Clearly, $R$ is again an orthogonal projection in $\cH$ since $K$ is self-adjoint and unitary. Taking into
 account that $K=I_{\cH}-2Q^\perp=2Q-I_{\cH}$, one observes that
 \begin{equation}\label{eq:PQpPQP}
  \begin{aligned}
   4PQ^\perp PQP &= -4PQ^\perp P^\perp QP = P\bigl(I_{\cH}-2Q^\perp\bigr)P^\perp\bigl(2Q-I_{\cH}\bigr)P\\
   &= P KP^\perp K P = PR^\perp P
  \end{aligned}
 \end{equation}
 and, similarly, that
 \begin{equation}\label{eq:PpQPpQpPp}
  4P^\perp Q P^\perp Q^\perp P^\perp = P^\perp R P^\perp\,.
 \end{equation}
 Combining \eqref{eq:sin2ThetaSqr}--\eqref{eq:PpQPpQpPp} yields
 \[
  \sin^2\bigl(2\Theta(P,Q)\bigr) = PR^\perp P + P^\perp RP^\perp = S(P,R) = \sin^2\bigl(\Theta(P,R)\bigr)\,,
 \]
 which proves the claim by taking the square roots.
 \end{proof}%
\end{lemma}

%%%%%%%%%%%%%%%%%%%%%%%%%%%%%%%%%%%%%%%%%%%%%%%%%%%%%%%%%%%%%%%%%%%%%%%%%%%%%%%%%%%%%%%%%%%%%%%%%%%%%%%%%%%%%%%%%%%%%%%%%%%%%%%%%%%
%%% Symmetric \sin\Theta theorem
%%%%%%%%%%%%%%%%%%%%%%%%%%%%%%%%%%%%%%%%%%%%%%%%%%%%%%%%%%%%%%%%%%%%%%%%%%%%%%%%%%%%%%%%%%%%%%%%%%%%%%%%%%%%%%%%%%%%%%%%%%%%%%%%%%%
The preceding Lemma \ref{lem:sin2ThetasinTheta} motivates to consider suitable bounds on $\sin\Theta$. In this respect, the
following proposition is essential for our considerations.

\begin{proposition}[The symmetric $\sin\Theta$ theorem]\label{prop:symmSinTheta}
 Let $A$ be a self-adjoint operator on a separable Hilbert space $\cH$, and let $V\in\cL(\cH)$ be self-adjoint. Write
 \[
  \spec(A) = \sigma \cup \Sigma \quad\text{ and }\quad \spec(A+V)=\omega\cup\Omega
 \]
 with $\sigma\cap\Sigma=\emptyset=\omega\cap\Omega$, and suppose that there is $d>0$ such that
 \[
  \dist(\sigma,\Omega) \ge d \quad\text{ and }\quad \dist(\Sigma,\omega) \ge d\,.
 \]
 Then, the operator angle $\Theta=\Theta(\EE_A(\sigma),\EE_{A+V}(\omega))$ satisfies the bound
 \[
  \norm{\sin\Theta} = \norm{\EE_A(\sigma)-\EE_{A+V}(\omega)} \le \frac{\pi}{2}\,\frac{\norm{V}}{d}\,.
 \]
\end{proposition}

The proof of Proposition \ref{prop:symmSinTheta} mainly relies on the following well-known result.
\begin{proposition}[see, e.g., {\cite[Proposition 3.4]{AM12:2}}]\label{prop:sinTheta0}
 Let $A$ be a self-adjoint operator on a separable Hilbert space $\cH$. Moreover, let $V\in\cL(\cH)$ be self-adjoint, and let
 $\delta,\Delta\subset\R$ be two Borel sets. Then
 \[
  \dist(\delta,\Delta)\norm{\EE_A(\delta)\EE_{A+V}(\Delta)} \le \frac{\pi}{2}\,\norm{\EE_A(\delta)V\EE_{A+V}(\Delta)}\,.
 \]
\end{proposition}

\begin{proof}[Proof of Proposition \ref{prop:symmSinTheta}]
 Proposition \ref{prop:sinTheta0} implies that
 \[
  \norm{\EE_A(\sigma)\EE_{A+V}(\Omega)} \le \frac{\pi}{2}\, \frac{\norm{\EE_{A}(\sigma)V\EE_{A+V}(\Omega)}}{d}
  \le \frac{\pi}{2}\, \frac{\norm{V}}{d}
 \]
 and, in the same way, that
 \[
  \norm{\EE_A(\Sigma)\EE_{A+V}(\omega)} \le \frac{\pi}{2}\,\frac{\norm{\EE_{A}(\Sigma)V\EE_{A+V}(\omega)}}{d}
  \le \frac{\pi}{2}\, \frac{\norm{V}}{d}\,.
 \]
 The claim now follows from \eqref{eq:projDiffNorm} and the identity
 \begin{equation}\label{eq:projDiffNormMax}
  \norm{\EE_A(\sigma)-\EE_{A+V}(\omega)} = \max\{\norm{\EE_A(\sigma)\EE_{A+V}(\Omega)}, \norm{\EE_A(\Sigma)\EE_{A+V}(\omega)}\}\,,
 \end{equation}
 which is due to \eqref{eq:projDiff}. Note that $\EE_A(\sigma)^\perp=\EE_A(\Sigma)$ and
 $\EE_{A+V}(\omega)^\perp=\EE_{A+V}(\Omega)$.
\end{proof}%
An extension of Proposition \ref{prop:symmSinTheta} to symmetrically-normed ideals is discussed in more detail in Section
\ref{sec:snIdeals} below.

\begin{remark}\label{rem:symmSinTheta}
 Proposition \ref{prop:symmSinTheta} is a variant of the symmetric $\sin\Theta$ theorem from \cite[Proposition 6.1]{DK70}. There,
 for each of the pairs $(\sigma,\Omega)$ and $(\Sigma,\omega)$ it is additionally assumed that the convex hull of one of the sets
 is disjoint from the other set. As a consequence, instead of $\frac{\pi}{2}$, the constant $1$ appears in the conclusion of
 \cite[Proposition 6.1]{DK70}, cf.\ also Remark \ref{rem:Sylvester} below. Note that, in contrast to Theorem \ref{thm:intro},
 Proposition \ref{prop:symmSinTheta} requires information both on $\spec(A)$ and $\spec(A+V)$.

 Although not stated in this explicit way, Proposition \ref{prop:symmSinTheta} is present in several recent works. For example, in
 the case where the operator $A$ is assumed to be bounded, it is used to prove \cite[Theorem 1]{KMM03} and \cite[Theorem 1]{KMM07}.
 In the unbounded setting, it appears, for instance, in the proof of \cite[Theorem 3.5]{AM12:2}.
\end{remark}

%%%%%%%%%%%%%%%%%%%%%%%%%%%%%%%%%%%%%%%%%%%%%%%%%%%%%%%%%%%%%%%%%%%%%%%%%%%%%%%%%%%%%%%%%%%%%%%%%%%%%%%%%%%%%%%%%%%%%%%%%%%%%%%%%%%
%%% Proof \sin2\Theta theorem
%%%%%%%%%%%%%%%%%%%%%%%%%%%%%%%%%%%%%%%%%%%%%%%%%%%%%%%%%%%%%%%%%%%%%%%%%%%%%%%%%%%%%%%%%%%%%%%%%%%%%%%%%%%%%%%%%%%%%%%%%%%%%%%%%%%
Let us recall that a closed subspace $\cU\subset\cH$ is called \emph{invariant} for a linear operator $B$ in $\cH$ if $B$ maps the
intersection $\Dom(B)\cap \cU$ into $\cU$.

The subspace $\cU$ is called \emph{reducing} for $B$ if both $\cU$ and its orthogonal complement $\cU^\perp$ are invariant for $B$
and the domain $\Dom(B)$ splits as
\begin{equation}\label{eq:domSplitting}
 \Dom(B) = \bigl(\Dom(B)\cap \cU\bigr) + \bigl(\Dom(B)\cap \cU^\perp\bigr)\,.
\end{equation}
Clearly, $\cU$ is reducing for $B$ if and only if $\cU^\perp$ is.

We are now able to proof the main result of this work.

\begin{proof}[Proof of Theorem \ref{thm:intro}]
 In essence, we follow the proof in \cite[Section 7]{DK70}.
 
 As in Lemma \ref{lem:sin2ThetasinTheta}, let $K$ denote the self-adjoint unitary operator on $\cH$ given by
 \[
  K := Q-Q^\perp\,.
 \]
 Since $\Ran Q$ is reducing for $A+V$, the splitting property \eqref{eq:domSplitting} implies that $K$ maps $\Dom(A+V)=\Dom(A)$
 onto itself. It also follows from \eqref{eq:domSplitting} and the invariance of the subspaces $\Ran Q$ and $\Ran Q^\perp$ that
 $K(A+V)Kx=(A+V)x$ for $x\in\Dom(A)$, so that $K(A+V)K=A+V$. The operator
 \begin{equation}\label{eq:defD}
  D:=KAK \quad\text{ on }\quad \Dom(D):=\Dom(A)
 \end{equation}
 is therefore self-adjoint and satisfies
 \begin{equation}\label{eq:DpertA}
  D=K(A+V)K - KVK = A+V - KVK\,.
 \end{equation}
 Clearly, the spectra of $A$ and $D$ coincide, that is,
 \[
  \spec(D) = \spec(A) = \sigma \cup \Sigma\,.
 \]
 In particular, one has
 \begin{equation}\label{eq:specProj}
  \EE_D(\sigma)=K\EE_A(\sigma)K \quad\text{ and }\quad \EE_D(\Sigma)= K\EE_A(\Sigma)K\,.
 \end{equation}
 
 Considering $D$ by \eqref{eq:DpertA} as a perturbation of $A$, and taking into account that $\dist(\sigma,\Sigma)=d>0$, it now
 follows from Proposition \ref{prop:symmSinTheta} that
 \[
  \norm{\sin\bigl(\Theta(\EE_A(\sigma),\EE_D(\sigma))\bigr)}
  \le \frac{\pi}{2}\,\frac{\norm{V-KVK}}{d}
  \le \frac{\pi}{2}\cdot 2\,\frac{\norm{V}}{d}\,,
 \]
 where the last inequality is due to the fact that $\norm{KVK}=\norm{V}$ since $K$ is unitary. In view of \eqref{eq:specProj} and
 Lemma \ref{lem:sin2ThetasinTheta}, this proves the claim.
\end{proof}%

%%%%%%%%%%%%%%%%%%%%%%%%%%%%%%%%%%%%%%%%%%%%%%%%%%%%%%%%%%%%%%%%%%%%%%%%%%%%%%%%%%%%%%%%%%%%%%%%%%%%%%%%%%%%%%%%%%%%%%%%%%%%%%%%%%%
%%% Discussion of \sin2\Theta theorem and proof of Corollary 2
%%%%%%%%%%%%%%%%%%%%%%%%%%%%%%%%%%%%%%%%%%%%%%%%%%%%%%%%%%%%%%%%%%%%%%%%%%%%%%%%%%%%%%%%%%%%%%%%%%%%%%%%%%%%%%%%%%%%%%%%%%%%%%%%%%%
If, in the situation of Theorem \ref{thm:intro}, it is known that $\theta:=\norm{\Theta}\le\frac{\pi}{4}$, then one has
$\norm{\sin2\Theta}=\sin2\theta$. In this case, taking into account \eqref{eq:projDiffNorm}, the bound \eqref{eq:introSin2Theta}
can equivalently be rewritten as
\begin{equation}\label{eq:maxAngleBound}
 \theta = \arcsin\bigl(\norm{\EE_A(\sigma)-Q}\bigr) \le \frac{1}{2}\arcsin\Bigl(\frac{\pi}{2}\cdot 2\,\frac{\norm{V}}{d}\Bigr)\,,
\end{equation}
see also \cite[Remark 4.2]{AM12:2}. The quantity $\theta$ is called the \emph{maximal angle} between the subspaces
$\Ran\EE_A(\sigma)$ and $\Ran Q$, see \cite[Definition 2.1]{AM12:2}.

However, the condition $\theta\le\frac{\pi}{4}$ does not need to be satisfied for arbitrary reducing subspaces for $A+V$, even if
the perturbation $V$ is small in norm. In fact, although the spectrum of $\Theta$ is known to have a gap around $\frac{\pi}{4}$
whenever $\norm{V}<\frac{d}{\pi}$, the following observation illustrates that the operator angle $\Theta$ may a priori have
spectrum everywhere else in the interval $\bigl[0,\frac{\pi}{2}\bigr]$.

\begin{remark}\label{rem:opAngle}
 In addition to the hypotheses of Theorem \ref{thm:intro}, assume that $\norm{V}<\frac{d}{\pi}$ and that
 $\norm{\Theta(P,Q)}<\frac{\pi}{4}$, where $P:=\EE_A(\sigma)$. The estimate \eqref{eq:maxAngleBound} then implies that
 \begin{equation}\label{eq:specThetaPQ}
  \spec\bigl(\Theta(P,Q)\bigr)\subset [0,\alpha] \quad\text{ with }\quad
  \alpha:=\frac{1}{2}\arcsin\Bigl(\frac{\pi}{2}\cdot 2\,\frac{\norm{V}}{d}\Bigr) < \frac{\pi}{4}\,.
 \end{equation}

 Since $S(P,Q^\perp)=C(P,Q)$ and, therefore, $\sin\bigl(\Theta(P,Q^\perp)\bigr) = \cos\bigl(\Theta(P,Q)\bigr)$, it follows from
 \eqref{eq:specThetaPQ} that
 \begin{equation}\label{eq:specThetaPQp}
  \spec\bigl(\Theta(P,Q^\perp)\bigr) \subset \Bigl[\frac{\pi}{2}-\alpha,\frac{\pi}{2}\Bigr]\,.
 \end{equation}

 Now, suppose that $R$ is an orthogonal projection onto a reducing subspace for $A+V$ such that
 \[
  \Ran R \cap \Ran Q \neq \{0\} \neq \Ran R\cap \Ran Q^\perp\,.
 \]
 Let $x\in\Ran R\cap\Ran Q$ with $\norm{x}=1$. Using the identity $(P-R)x=(P-Q)x$ and the inclusion \eqref{eq:specThetaPQ}, one
 observes that
 \begin{equation}\label{eq:numThetaPR1}
  \begin{aligned}
   \langle x,\sin^2\bigl(\Theta(P,R)\bigr)x\rangle &= \langle x, (P-R)^2x\rangle = \langle x,(P-Q)^2x\rangle\\
   &= \langle x,\sin^2\bigl(\Theta(P,Q)\bigr)x\rangle \le \sin^2\alpha\,.
  \end{aligned}
 \end{equation}
 Taking into account \eqref{eq:specThetaPQp}, for $y\in\Ran R\cap\Ran Q^\perp$, $\norm{y}=1$, one obtains in a similar way that
 \begin{equation}\label{eq:numThetaPR2}
  \langle y,\sin^2\bigl(\Theta(P,R)\bigr)y\rangle = \langle y, \sin^2\bigl(\Theta(P,Q^\perp)\bigr)y\rangle
  \ge \sin^2\Bigl(\frac{\pi}{2}-\alpha\Bigr)\,.
 \end{equation}
 Combining \eqref{eq:numThetaPR1} and \eqref{eq:numThetaPR2} yields that $\Theta(P,R)$ has spectrum both in $[0,\alpha]$ and
 $\bigl[\frac{\pi}{2}-\alpha,\frac{\pi}{2}\bigr]$.

 Thus, depending on the reducing subspace for $A+V$ that is considered, the operator angle has spectrum in $[0,\alpha]$,
 $\bigl[\frac{\pi}{2}-\alpha,\frac{\pi}{2}\bigr]$, or both.
\end{remark}

In the situation of Corollary \ref{cor:intro}, the projection $Q$ is chosen very specifically, namely
$Q=\EE_{A+V}\bigl(\cO_{d/2}(\sigma)\bigr)$. Provided that $\norm{V}<\frac{d}{2}$, the spectral subspace
$\Ran \EE_{A+V}\bigl(\cO_{d/2}(\sigma)\bigr)$ can be regarded as the perturbation of the spectral subspace $\Ran \EE_A(\sigma)$ for
the unperturbed operator $A$, see \cite{AM12:2}. It turns out that, in this case, the condition $\theta\le\frac{\pi}{4}$ is
automatically satisfied whenever $\norm{V}\le\frac{d}{\pi}$. Indeed, the mapping
$[0,1]\ni t\mapsto\EE_{A+tV}\bigl(\cO_{d/2}(\sigma)\bigr)$ is norm continuous, see \cite[Theorem 3.5]{AM12:2}; in fact, this
follows from the symmetric $\sin\Theta$ theorem. Corollary \ref{cor:intro} is therefore a direct consequence of the following more
general statement.

\begin{lemma}\label{lem:cor}
 Let $A$, $V$, and $Q$ be as in Theorem \ref{thm:intro}, and suppose that $\norm{V}\le\frac{d}{\pi}$. If there is a norm continuous
 path $[0,1]\ni t\mapsto P_t$ of orthogonal projections in $\cH$ with $P_0=\EE_A(\sigma)$ and $P_1=Q$ such that $\Ran P_t$ is
 reducing for $A+tV$ for all $t\in[0,1]$, then
 \[
  \arcsin\bigl(\norm{\EE_A(\sigma)-Q}\bigr) \le \frac{1}{2}\arcsin\Bigl(\frac{\pi}{2}\cdot 2\,\frac{\norm{V}}{d}\Bigr)
  \le \frac{\pi}{4}\,.
 \]
 
 \begin{proof}
  In view of Theorem \ref{thm:intro} (or more precisely, estimate \eqref{eq:maxAngleBound}), it suffices to show the inequality
  \begin{equation}\label{eq:anglePi4}
   \arcsin\bigl(\norm{\EE_A(\sigma)-Q}\bigr)\le\frac{\pi}{4}\,.
  \end{equation}
  
  Assume that \eqref{eq:anglePi4} does not hold. Then, since the path $[0,1]\ni t\mapsto P_t$ is assumed to be norm continuous with
  $P_0=\EE_A(\sigma)$ and $P_1=Q$, there is $\tau\in(0,1)$ such that
  \begin{equation}\label{eq:maxAnglePi4}
   \arcsin\bigl(\norm{\EE_A(\sigma)-P_\tau}\bigr) = \frac{\pi}{4}\,.
  \end{equation}
  On the other hand, taking into account that $\Ran P_\tau$ is reducing for $A+\tau V$ and that $\tau\norm{V}<\frac{d}{\pi}$, it
  follows from inequality \eqref{eq:maxAngleBound} that
  \[
   \arcsin\bigl(\norm{\EE_A(\sigma)-P_\tau}\bigr) \le \frac{1}{2}\arcsin\Bigl(\frac{\pi}{2}\cdot 2\,\frac{\norm{\tau V}}{d}\Bigr)
   < \frac{\pi}{4}\,,
  \]
  which is a contradiction to \eqref{eq:maxAnglePi4}. This shows inequality \eqref{eq:anglePi4}.
 \end{proof}%
\end{lemma}

\begin{remark}
 The bound \eqref{eq:introMaxAngleBound} from Corollary \ref{cor:intro} has already been mentioned in \cite[Remark 4.4]{AM12:2},
 but only for the particular case of perturbations $V$ satisfying $\norm{V}\le\frac{\ee-1}{2\ee}\,d$, where
 $\frac{4}{\pi^2+4}<\frac{\ee-1}{2\ee}<\frac{1}{\pi}$. For those perturbations, the condition $\theta\le\frac{\pi}{4}$ has been
 ensured by use of other known bounds on $\theta=\norm{\Theta}$.
\end{remark}

%%%%%%%%%%%%%%%%%%%%%%%%%%%%%%%%%%%%%%%%%%%%%%%%%%%%%%%%%%%%%%%%%%%%%%%%%%%%%%%%%%%%%%%%%%%%%%%%%%%%%%%%%%%%%%%%%%%%%%%%%%%%%%%%%%%
%%% \sin2\Theta and Corollary 2: Davis-Kahan case
%%%%%%%%%%%%%%%%%%%%%%%%%%%%%%%%%%%%%%%%%%%%%%%%%%%%%%%%%%%%%%%%%%%%%%%%%%%%%%%%%%%%%%%%%%%%%%%%%%%%%%%%%%%%%%%%%%%%%%%%%%%%%%%%%%%
We close this section with a discussion of Theorem \ref{thm:intro} and Corollary \ref{cor:intro} under the additional spectral
separation conditions from \cite{DK70}.

\begin{remark}\label{rem:sin2Theta}
 In addition to the hypotheses of Theorem \ref{thm:intro}, assume that the convex hull of one of the sets $\sigma$ and $\Sigma$ is
 disjoint from the other set. In this case, the constant $\frac{\pi}{2}$ in the bound \eqref{eq:introSin2Theta} can be replaced by
 $1$, see Remark \ref{rem:symmSinTheta}. The resulting estimate is the bound from the Davis-Kahan $\sin2\Theta$ theorem in
 \cite{DK70}, that is,
 \[
  \norm{\sin2\Theta} \le 2\,\frac{\norm{V}}{d}\,.
 \]
 For the particular case of $Q=\EE_{A+V}\bigl(\cO_{d/2}(\sigma)\bigr)$, as in Corollary \ref{cor:intro} this bound can equivalently
 be rewritten as
 \begin{equation}\label{eq:DKsin2ThetamaxAngle}
  \arcsin\bigl(\norm{\EE_A(\sigma)-\EE_{A+V}\bigl(\cO_{d/2}(\sigma)\bigr)}\bigr)
  \le \frac{1}{2}\arcsin\Bigl(2\,\frac{\norm{V}}{d}\Bigr) < \frac{\pi}{4}\,,
 \end{equation}
 whenever $\norm{V}<\frac{d}{2}$. It has already been stated by Davis in \cite[Theorem 5.1]{Davis63} that this estimate is sharp in
 the sense that equality can be attained. This can be seen from the following example of $2\times2$ matrices: Let
 \[
  A := \begin{pmatrix} 1 & 0\\ 0 & -1 \end{pmatrix}\quad\text{ with }\quad \sigma:=\{1\}\quad\text{ and }\quad \Sigma:=\{-1\}\,.
 \]
 Obviously, one has $d:=\dist(\sigma,\Sigma)=2$. For arbitrary $x$ with $0<x<1=\frac{d}{2}$ consider
 \[
  V := \begin{pmatrix} -x^2 & x\sqrt{1-x^2}\\ x\sqrt{1-x^2} & x^2 \end{pmatrix}.
 \]
 It is easy to verify that $\norm{V}=x$ and that $\spec(A+V)=\{\pm\sqrt{1-x^2}\}$.
 
 Denote $\theta:=\frac{1}{2}\arcsin(x)<\frac{\pi}{4}$. Then, one has
 \begin{equation}\label{eq:DKSharpTanCot}
  \frac{1-\sqrt{1-x^2}}{x} = \frac{1-\cos(2\theta)}{\sin(2\theta)}=\tan\theta\quad\text{ and }\quad
  \frac{1+\sqrt{1-x^2}}{x}=\cot\theta\,.
 \end{equation}
 Using \eqref{eq:DKSharpTanCot}, a straightforward computation shows that
 \[
  U^*(A+V)U = \begin{pmatrix} \sqrt{1-x^2} & 0\\ 0 & -\sqrt{1-x^2} \end{pmatrix}\quad\text{ where }\quad
  U = \begin{pmatrix} \cos\theta & -\sin\theta\\ \sin\theta & \cos\theta \end{pmatrix}.
 \]
 In particular, this implies that
 \[
  \EE_{A+V}\bigl(\cO_1(\sigma)\bigr)
  = \begin{pmatrix} \cos\theta\\ \sin\theta \end{pmatrix}  \begin{pmatrix} \cos\theta & \sin\theta \end{pmatrix}
  = \begin{pmatrix} \cos^2\theta & \sin\theta\cos\theta\\ \sin\theta\cos\theta & \sin^2\theta\end{pmatrix},
 \]
 so that
 \[
  \arcsin\bigl(\norm{\EE_A(\sigma)-\EE_{A+V}\bigl(\cO_1(\sigma)\bigr)}\bigr)=\theta=\frac{1}{2}\arcsin(x)
  =\frac{1}{2}\arcsin\Bigl(2\,\frac{\norm{V}}{d}\Bigr)\,.
 \]
 Hence, inequality \eqref{eq:DKsin2ThetamaxAngle} is sharp.
\end{remark}

%%%%%%%%%%%%%%%%%%%%%%%%%%%%%%%%%%%%%%%%%%%%%%%%%%%%%%%%%%%%%%%%%%%%%%%%%%%%%%%%%%%%%%%%%%%%%%%%%%%%%%%%%%%%%%%%%%%%%%%%%%%%%%%%%%%
%%%%%%%%%%%%%%%%%%%%%%%%%%%%%%%%%%%%%%%%%%%%%%%%%%%%%%%%%%%%%%%%%%%%%%%%%%%%%%%%%%%%%%%%%%%%%%%%%%%%%%%%%%%%%%%%%%%%%%%%%%%%%%%%%%%
%%% Section: sin 2\theta estimate
%%%%%%%%%%%%%%%%%%%%%%%%%%%%%%%%%%%%%%%%%%%%%%%%%%%%%%%%%%%%%%%%%%%%%%%%%%%%%%%%%%%%%%%%%%%%%%%%%%%%%%%%%%%%%%%%%%%%%%%%%%%%%%%%%%%
%%%%%%%%%%%%%%%%%%%%%%%%%%%%%%%%%%%%%%%%%%%%%%%%%%%%%%%%%%%%%%%%%%%%%%%%%%%%%%%%%%%%%%%%%%%%%%%%%%%%%%%%%%%%%%%%%%%%%%%%%%%%%%%%%%%
\section{The generic $\sin2\theta$ estimate}\label{sec:sin2theta}

In this section, we present an alternative, straightforward proof of the generic $\sin2\theta$ estimate \eqref{eq:introSin2theta}
that uses a different technique than the one presented for Theorem \ref{thm:intro} and, at the same time, is more direct than the
one in \cite{AM12:2}.

It is worth mentioning that the inequality \eqref{eq:maxAngleBound}, and therefore also Corollary \ref{cor:intro}, can be deduced
from the estimate \eqref{eq:introSin2theta} as well since $\norm{\sin2\Theta}=\sin2\theta$ whenever
$\theta=\norm{\Theta}\le\frac{\pi}{4}$. An immediate advantage of the $\sin2\theta$ estimate is that it can be formulated without
the notion of the operator angle, see Proposition \ref{prop:sin2theta} below.

%%%%%%%%%%%%%%%%%%%%%%%%%%%%%%%%%%%%%%%%%%%%%%%%%%%%%%%%%%%%%%%%%%%%%%%%%%%%%%%%%%%%%%%%%%%%%%%%%%%%%%%%%%%%%%%%%%%%%%%%%%%%%%%%%%%
%%% Sylvester equation
%%%%%%%%%%%%%%%%%%%%%%%%%%%%%%%%%%%%%%%%%%%%%%%%%%%%%%%%%%%%%%%%%%%%%%%%%%%%%%%%%%%%%%%%%%%%%%%%%%%%%%%%%%%%%%%%%%%%%%%%%%%%%%%%%%%
Given two self-adjoint operators $B_0$ and $B_1$ on Hilbert spaces $\cH_0$ and $\cH_1$, respectively, recall that a bounded
operator $Y\in\cL(\cH_0,\cH_1)$ is called a \emph{strong solution to the operator Sylvester equation}
\begin{equation}\label{eq:defSylEq}
 YB_0 - B_1Y = T\,,\quad T\in\cL(\cH_0,\cH_1)\,, 
\end{equation}
if
\[
 \Ran(Y|_{\Dom(B_0)}) \subset \Dom(B_1)
\]
and
\begin{equation}\label{eq:strongSylvester}
 YB_0g - B_1Yg = Tg \quad\text{ for }\quad g\in\Dom(B_0)\,.
\end{equation}

We need the following well-known result, which also plays a crucial role for the extension of the symmetric $\sin\Theta$ theorem in
Section \ref{sec:snIdeals}, see Proposition \ref{prop:symmSinThetasnIdeals} below.

\begin{theorem}\label{thm:Sylvester}
 Let $B_0$ and $B_1$ be two self-adjoint operators on Hilbert spaces $\cH_0$ and $\cH_1$, respectively, such that
 \begin{equation}\label{eq:SylvSpecSep}
  d:=\dist\bigl(\spec(B_0),\spec(B_1)\bigr) > 0\,.
 \end{equation}
 Then, the Sylvester equation \eqref{eq:defSylEq} has a unique strong solution $Y\in\cL(\cH_0,\cH_1)$. This solution admits
 \begin{equation}\label{eq:SylSol}
  \langle h, Yg \rangle = \int_\R  \langle h, \ee^{\ii tB_1}T\ee^{-\ii tB_0}g\rangle f_d(t)\,\dd t
  \quad\text{ for }\quad g\in\cH_0\,,\ h\in\cH_1\,,
 \end{equation}
 where $f_d$ is any function in $L^1(\R)$, continuous except at zero, such that
 \[
  \hat{f_d}(\lambda):=\int_\R \ee^{-\ii t\lambda} f_d(t)\,\dd t = \frac{1}{\lambda} \quad\text{ whenever }\quad
  \abs{\lambda}\ge d\,.
 \]
 In particular, $Y$ satisfies the norm bound
 \begin{equation}\label{eq:normBoundSylv}
  \norm{Y} \le c\,\frac{\norm{T}}{d}\,,
 \end{equation}
 where
 \begin{equation}\label{eq:Nagy}
  c = \inf\Bigl\{ \norm{f}_{L^1(\R)} \,\Bigm|\, f\in L^1(\R)\,,\ \hat{f}(\lambda)=\frac{1}{\lambda}\ \text{ whenever }\
      \abs{\lambda}\ge 1\Bigr\}
  = \frac{\pi}{2}\,,
 \end{equation}
 and this constant is sharp in \eqref{eq:normBoundSylv}.
 \begin{proof}
  This is obtained by combining \cite[Theorem 2.7]{AMM03} and \cite[Lemma 4.2]{AM11}; cf.\ also \cite[Remark 2.8]{AMM03},
  \cite[Theorem 3.2]{AM12:2}, and \cite[Theorem 4.1]{BDM83}. Note that the last equality in \eqref{eq:Nagy} goes back to Sz.-Nagy
  and Strausz \cite{SzNagy53}, \cite{SzNagy87}. The fact that the constant $c=\frac{\pi}{2}$ in \eqref{eq:normBoundSylv} is sharp
  is due to McEachin \cite{McEachin92}.
 \end{proof}%
\end{theorem}

\begin{remark}\label{rem:Sylvester}
 A statement analogous to Theorem \ref{thm:Sylvester} holds if the operators $B_0$ and $B_1$ are assumed to be just normal and
 their spectra are separated as in \eqref{eq:SylvSpecSep}. In this case, the solution to \eqref{eq:defSylEq} admits a
 representation similar to \eqref{eq:SylSol}, and the constant $c$ in \eqref{eq:normBoundSylv} has to be replaced by some constant
 less than $2{.}91$, see \cite{BDK89} and \cite[Theorem 4.2]{BDM83}. Note that the exact value of the optimal constant here is
 still unknown.

 However, in some cases a better constant is available. If, for example, the spectra of $B_0$ and $B_1$ are additionally assumed to
 be subordinated in the sense that they are contained in half planes $\Pi_0$ and $\Pi_1$, respectively, such that
 $\dist(\Pi_0,\Pi_1)\ge d$, or if one of the two sets is contained in a disk of finite radius with distance at least $d$ from the
 other one, then the constant in the bound \eqref{eq:normBoundSylv} can be replaced by $1$. This follows from corresponding
 representation formulae for the solution, see \cite[Theorem 3.3]{BDM83} and \cite[Theorem 9.1]{BR97}; cf.\ also
 \cite[Theorem 3.1]{BDM83}. Other improvements on the constant may be available for small dimensions of the underlying Hilbert
 spaces, at least in the case where $B_0$ and $B_1$ are assumed to be self-adjoint, see \cite{McEachin94}.
\end{remark}

%%%%%%%%%%%%%%%%%%%%%%%%%%%%%%%%%%%%%%%%%%%%%%%%%%%%%%%%%%%%%%%%%%%%%%%%%%%%%%%%%%%%%%%%%%%%%%%%%%%%%%%%%%%%%%%%%%%%%%%%%%%%%%%%%%%
%%% sin2\theta theorem
%%%%%%%%%%%%%%%%%%%%%%%%%%%%%%%%%%%%%%%%%%%%%%%%%%%%%%%%%%%%%%%%%%%%%%%%%%%%%%%%%%%%%%%%%%%%%%%%%%%%%%%%%%%%%%%%%%%%%%%%%%%%%%%%%%%
We now give a straightforward proof of the $\sin2\theta$ estimate \eqref{eq:introSin2theta}. One immediately visible difference to
the proof in \cite{AM12:2} is that the proof given below is direct and is not deduced from the corresponding a posteriori estimate.
In addition, the key idea of the argument presented here can easily be reduced to one single equation, namely equation
\eqref{eq:strongRiccatiAsSyl} below, which makes this proof very transparent.

\begin{proposition}[{\cite[Corollary 4.3]{AM12:2}}]\label{prop:sin2theta}
Let $A$, $V$, and $Q$ be as in Theorem \ref{thm:intro}. Then
\[
 \sin2\theta \le \frac{\pi}{2}\cdot 2\,\frac{\norm{V}}{d}\,,
\]
where $\theta:=\arcsin\bigl(\norm{\EE_A(\sigma)-Q}\bigr)$ denotes the maximal angle between the subspaces $\Ran\EE_A(\sigma)$ and
$\Ran Q$.

\begin{proof}
 The case $\theta=\frac{\pi}{2}$ is obvious. Assume that $\theta<\frac{\pi}{2}$, that is,
 \begin{equation}\label{eq:acuteAngleCase}
  \norm{\EE_A(\sigma)-Q}<1\,.
 \end{equation}
 Denote $\cH_0:=\Ran\EE_A(\sigma)$ and $\cH_1:=\cH_0^\perp=\Ran\EE_A(\Sigma)$, and let
 \[
  V = \begin{pmatrix} V_0 & W\\ W^* & V_1 \end{pmatrix} \quad\text{and }\quad A = \begin{pmatrix} A_0 & 0\\ 0 & A_1\end{pmatrix}
 \]
 with $\Dom(A)=\Dom(A_0)\oplus\Dom(A_1)$ be the representations of $V$ and $A$ as $2\times 2$ block operator matrices with respect
 to the decomposition $\cH=\cH_0\oplus\cH_1$.
   
 It is well known (see, e.g., \cite[Corollary 3.4 (i)]{KMM03:2}) that under the condition \eqref{eq:acuteAngleCase} there is a
 unique operator $X\in\cL(\cH_0,\cH_1)$ such that the range of $Q$ is the graph of $X$, that is,
 $\Ran Q=\{x\oplus Xx\mid x\in \cH_0\}$. This operator $X$ satisfies
 \begin{equation}\label{eq:theta}
  \arctan(\norm{X}) = \arcsin\bigl(\norm{\EE_A(\sigma)-Q}\bigr) = \theta\,.
 \end{equation}
 Moreover, the operator $U\in\cL(\cH)$ given by
 \[
  U = \begin{pmatrix}
       \left(I_{\cH_0}+X^*X\right)^{-1/2} & -X^*\left(I_{\cH_1}+XX^*\right)^{-1/2}\\
       X\left(I_{\cH_0}+X^*X\right)^{-1/2} & \left(I_{\cH_1}+XX^*\right)^{-1/2}
      \end{pmatrix}
 \]
 is unitary and satisfies $U^*\EE_{A+V}\bigl(\cO_{d/2}(\sigma)\bigr)U=\EE_A(\sigma)$, cf.\ \cite[Remark 3.6]{KMM03:2}.

 Considering $\Dom(A_0+V_0)=\Dom(A_0)$, $\Dom(A_1+V_1)=\Dom(A_1)$, and
 \[
  A+V = \begin{pmatrix} A_0+V_0 & 0\\ 0 & A_1+V_1 \end{pmatrix} + \begin{pmatrix} 0 & W\\ W^* & 0 \end{pmatrix}\,,
 \]
 it follows from \cite[Theorem 4.1]{MSS13} (see also \cite[Lemma 5.3]{AMM03}) that $X$ is a \emph{strong solution to the operator
 Riccati equation}
 \[
  X(A_0+V_0)-(A_1+V_1)X+XWX-W^*=0\,,
 \]
 that is,
 \[
  \Ran\bigl(X|_{\Dom(A_0)}\bigr)\subset \Dom(A_1)
 \]
 and
 \begin{equation}\label{eq:strongRiccati}
  X(A_0+V_0)g - (A_1+V_1)Xg + XWXg - W^*g = 0 \quad\text{for}\quad g\in\Dom(A_0)\,.
 \end{equation}

 Define $H\in\cL(\cH)$ by
 \[
  H := \begin{pmatrix} \left(I_{\cH_0}+X^*X\right)^{-1/2} & 0\\ 0 & \left(I_{\cH_1}+XX^*\right)^{-1/2}\end{pmatrix}\,.
 \]
 A straightforward calculation shows that
 \begin{equation}\label{eq:VunitaryTransf}
  U^*VU = H \begin{pmatrix} * & *\\ V_1X - XV_0 - XWX + W^* & * \end{pmatrix} H\,.
 \end{equation}
  
 Denote $P:=\EE_A(\sigma)$. Equations \eqref{eq:strongRiccati} and \eqref{eq:VunitaryTransf} then imply that
 \begin{equation}\label{eq:strongRiccatiAsSyl}
  \begin{split}
   XA_0g - A_1&Xg = V_1Xg - XV_0g - XWXg + W^*g\\
            &= \left(I_{\cH_1}+XX^*\right)^{1/2} \bigl(P^\perp U^*VUP|_{\cH_0}\bigr) \left(I_{\cH_0}+X^*X\right)^{1/2}g
  \end{split}
 \end{equation}
 for $g\in\Dom(A_0)$, where the restriction $P^\perp U^*VUP|_{\cH_0}$ is considered as an operator from $\cH_0$ to $\cH_1$.
 Comparing equation \eqref{eq:strongRiccatiAsSyl} with the Sylvester equation \eqref{eq:strongSylvester}, it follows from the bound
 in Theorem \ref{thm:Sylvester} given by \eqref{eq:normBoundSylv} and \eqref{eq:Nagy} that
 \[
  \norm{X} \le \frac{\pi}{2}\,\bigl(1+\norm{X}^2\bigr) \frac{\norm{P^\perp U^*VUP}}{d}
  \le \frac{\pi}{2}\,\bigl(1+\norm{X}^2\bigr)\frac{\norm{V}}{d}\,.
 \]
 Since $2\norm{X}/(1+\norm{X}^2)=2\tan\theta/(1+\tan^2\theta)=\sin(2\theta)$ by \eqref{eq:theta}, this proves the claim.
\end{proof}%
\end{proposition}

%%%%%%%%%%%%%%%%%%%%%%%%%%%%%%%%%%%%%%%%%%%%%%%%%%%%%%%%%%%%%%%%%%%%%%%%%%%%%%%%%%%%%%%%%%%%%%%%%%%%%%%%%%%%%%%%%%%%%%%%%%%%%%%%%%%
%%%%%%%%%%%%%%%%%%%%%%%%%%%%%%%%%%%%%%%%%%%%%%%%%%%%%%%%%%%%%%%%%%%%%%%%%%%%%%%%%%%%%%%%%%%%%%%%%%%%%%%%%%%%%%%%%%%%%%%%%%%%%%%%%%%
%%% Section: Symmetrically-normed ideals
%%%%%%%%%%%%%%%%%%%%%%%%%%%%%%%%%%%%%%%%%%%%%%%%%%%%%%%%%%%%%%%%%%%%%%%%%%%%%%%%%%%%%%%%%%%%%%%%%%%%%%%%%%%%%%%%%%%%%%%%%%%%%%%%%%%
%%%%%%%%%%%%%%%%%%%%%%%%%%%%%%%%%%%%%%%%%%%%%%%%%%%%%%%%%%%%%%%%%%%%%%%%%%%%%%%%%%%%%%%%%%%%%%%%%%%%%%%%%%%%%%%%%%%%%%%%%%%%%%%%%%%
\section{Symmetrically-normed ideals}\label{sec:snIdeals}

In this section, we extend Theorem \ref{thm:intro} to symmetrically-normed ideals of the algebra of bounded operators. In the
presentation of the concept of symmetrically-normed ideals we mainly follow \cite[Chapter III]{GK69}.

%%%%%%%%%%%%%%%%%%%%%%%%%%%%%%%%%%%%%%%%%%%%%%%%%%%%%%%%%%%%%%%%%%%%%%%%%%%%%%%%%%%%%%%%%%%%%%%%%%%%%%%%%%%%%%%%%%%%%%%%%%%%%%%%%%%
%%% Symmetrically-normed ideals
%%%%%%%%%%%%%%%%%%%%%%%%%%%%%%%%%%%%%%%%%%%%%%%%%%%%%%%%%%%%%%%%%%%%%%%%%%%%%%%%%%%%%%%%%%%%%%%%%%%%%%%%%%%%%%%%%%%%%%%%%%%%%%%%%%%
Let $\cH$ be a separable Hilbert space. Recall that a non-zero subspace $\fS\subset\cL(\cH)$ is called a \emph{two-sided ideal of
$\cL(\cH)$} if for every $T\in\fS$ and every choice of operators $K,T\in\cL(\cH)$ one has $KTL\in\fS$. It is well-known (see, e.g.,
\cite[Theorem III.1.1]{GK69}) that every two-sided ideal $\fS\subset\cL(\cH)$ contains the operators of finite rank and that either
$\fS=\cL(\cH)$ or $\fS\subset\fS_\infty$, where $\fS_\infty:=\fS_\infty(\cH)$ denotes the two-sided ideal of compact operators in
$\cL(\cH)$. At this point it should be emphasized that we allow for the case $\fS=\cL(\cH)$.

A norm $\tnorm{\cdot}$ on a two-sided ideal $\fS\subset\cL(\cH)$ is called \emph{symmetric} if it has the following properties:
\begin{enumerate}
 \renewcommand{\theenumi}{\roman{enumi}}
 \item $\tnorm{KTL}\le\norm{K}\,\tnorm{T}\,\norm{L}$ for $T\in\fS$ and $K,L\in\cL(\cH)$.
 \item $\tnorm{T}=\norm{T}$ if $T\in\fS$ has rank $1$.
\end{enumerate}
The ideal $\fS$ is called a \emph{symmetrically-normed ideal} if there is a symmetric norm $\tnorm{\cdot}$ on $\fS$ such that
$(\fS,\tnorm{\cdot})$ is complete.

Clearly, every symmetric norm $\tnorm{\cdot}$ on a two-sided ideal $\fS$ is \emph{unitary-invariant}, that is, for every unitary
operator $U\in\cL(\cH)$ one has
\[
 \tnorm{UT} = \tnorm{TU} = \tnorm{T}\,,\quad T\in\fS\,.
\]
Moreover, it follows by polar decomposition that for every $T\in\fS$ the operators $\abs{T}=\sqrt{T^*T}$, $T^*$, and $\abs{T^*}$
also belong to $\fS$ and
\begin{equation}\label{eq:snNormPolar}
 \tnorm{T} = \tnorm{\,\abs{T}\,} = \tnorm{T^*} = \tnorm{\,\abs{T^*}\,}\,.
\end{equation}

Examples for symmetric norms on every two-sided ideal $\fS$ are the \emph{Ky Fan norms} $\tnorm{\cdot}_n$, $n\in\N$, which are
defined as the sum of the first $n$ singular values, that is,
\begin{equation}\label{eq:defKyFan}
 \tnorm{T}_n := s_1(T) + \dots + s_n(T)\,,\quad T\in\cL(\cH)\,.
\end{equation}
Recall (see, e.g., \cite[Theorem II.7.1]{GK69}) that the $n$-th singular value of $T\in\cL(\cH)$ can be introduced as
\begin{equation}\label{eq:defSingularVal}
 s_n(T) := \inf\{ \norm{T-F} \mid F\in\cL(\cH)\,,\ \dim\Ran F < n\}\,,\quad n\in\N\,.
\end{equation}
In particular, one has $s_1(T)=\norm{T}$, and the sequence $(s_n(T))_n$ is non-increasing, hence convergent. Moreover,
$T\in\cL(\cH)$ is compact if and only if $(s_n(T))_n$ converges to zero, see \cite[Corollary II.7.1]{GK69}. In this case,
\eqref{eq:defSingularVal} agrees with the usual notion of singular values of compact operators, see \cite[Theorem II.2.1]{GK69}.

It follows from the discussion in \cite[Section II.7]{GK69} that the Ky Fan norms \eqref{eq:defKyFan} are indeed symmetric norms
and that each $\tnorm{\cdot}_n$ can be represented as
\begin{equation}\label{eq:KyFanSup}
 \tnorm{T}_n = \sup \biggl| \sum_{j=1}^n \langle y_j,Tx_j\rangle\biggr|\,,\quad T\in\cL(\cH)\,,
\end{equation}
where the supremum is taken over all orthonormal systems $\{x_j\}_{j=1}^n$ and $\{y_j\}_{j=1}^n$ in $\cH$, cf.\ also
\cite[Lemma II.4.1]{GK69}.

The Ky Fan norms play a very distinguished role in our considerations:
For a symmetrically-normed ideal $\fS$ with norm $\tnorm{\cdot}$, we say that $(\fS,\tnorm{\cdot})$ admits \emph{Ky Fan's dominance
theorem} if for $T\in\cL(\cH)$ and $S\in\fS$ with
\begin{equation}\label{eq:KyFanDom}
 \tnorm{T}_n\le\tnorm{S}_n \quad\text{ for all }\quad n\in\N
\end{equation}
one has $T\in\fS$ and $\tnorm{T}\le\tnorm{S}$. Note that \eqref{eq:KyFanDom} implies that $T$ is compact if $S$ is compact. Indeed,
in case of \eqref{eq:KyFanDom}, $(s_n(T))_n$ converges to zero if $(s_n(S))_n$ does.

It is shown in \cite{AA12} that every symmetric norm $\tnorm{\cdot}$ on $\fS=\cL(\cH)$ is equivalent to the usual bound norm on
$\cL(\cH)$ and that $(\cL(\cH),\tnorm{\cdot})$ admits Ky Fan's dominance theorem in the above sense.

In case of $\fS\subset\fS_\infty$, it is well-known that $(\fS,\tnorm{\cdot})$ admits Ky Fan's dominance theorem if
$(\fS,\tnorm{\cdot})$ is generated by a \emph{symmetric norming function}, see \cite[Section III.4]{GK69}. If
$\fS\subsetneq\fS_\infty$, this is the case if and only if for $(T_n)_n\subset\fS$ with $\sup_{n\in\N}\tnorm{T_n}<\infty$ and
$T_n\to T\in\cL(\cH)$ in the weak operator topology one has $T\in\fS$ and $\tnorm{T}\le\sup_{n\in\N}\tnorm{T_n}$,
cf.\ \cite[Theorems III.5.1 and III.5.2]{GK69}. This latter characterization has been used, for instance, in \cite{AMM03},
\cite{BDM83}, and \cite{McEachin93}.

Well-known examples of symmetrically-normed ideals satisfying Ky Fan's dominance theorem are the standard Schatten classes $\fS_p$
for $1\le p\le \infty$, see \cite[Section III.7]{GK69}.

%%%%%%%%%%%%%%%%%%%%%%%%%%%%%%%%%%%%%%%%%%%%%%%%%%%%%%%%%%%%%%%%%%%%%%%%%%%%%%%%%%%%%%%%%%%%%%%%%%%%%%%%%%%%%%%%%%%%%%%%%%%%%%%%%%%
%%% Symmetric sin\Theta theorem for symmetrically-normed ideals
%%%%%%%%%%%%%%%%%%%%%%%%%%%%%%%%%%%%%%%%%%%%%%%%%%%%%%%%%%%%%%%%%%%%%%%%%%%%%%%%%%%%%%%%%%%%%%%%%%%%%%%%%%%%%%%%%%%%%%%%%%%%%%%%%%%
In order to extend Theorem \ref{thm:intro} to symmetrically-normed ideals, one requires a suitable extension of the symmetric
$\sin\Theta$ theorem. A corresponding variant of Proposition \ref{prop:sinTheta0} is known in principle (see \cite{BDM83} and
\cite{McEachin93}), whereas the identity \eqref{eq:projDiffNormMax} does not hold for arbitrary norms. Nevertheless, one can
proceed as in \cite{DK70} and use \cite[Lemmas 6.1 and 6.2]{DK70} instead of \eqref{eq:projDiffNormMax}.

In the present work, we choose a direct way to extend the symmetric $\sin\Theta$ theorem. This approach, however, utilizes the
connection to the operator Sylvester equation just as well, so that Theorem \ref{thm:Sylvester} plays a crucial role in the our
reasoning too.

In contrast to the formulation of Proposition \ref{prop:symmSinTheta} and the hypotheses in \cite{BDM83} and \cite{McEachin93}, we
do not restrict ourselves to the case of spectral projections here. In this regard, let us recall that if $\cU$ is a reducing
subspace for a linear operator $B$, then the restrictions $B|_{\Dom(B)\cap\cU}$ and $B|_{\Dom(B)\cap\cU^\perp}$ are called the
\emph{parts of $B$} associated with $\cU$ and $\cU^\perp$, respectively. It is well-known that these parts are self-adjoint if $B$
is self-adjoint.

\begin{proposition}[The symmetric $\sin\Theta$ theorem for symmetrically-normed ideals]\label{prop:symmSinThetasnIdeals}
 Let $A$ be a self-adjoint operator on a separable Hilbert space $\cH$, let $V\in\cL(\cH)$ be self-adjoint, and suppose that $P$
 and $Q$ are orthogonal projections onto reducing subspaces for $A$ and $A+V$, respectively. Let $A_0$ and $A_1$ denote the parts
 of $A$ associated with $\Ran P$ and $\Ran P^\perp$, respectively, and let $\Lambda_0$ and $\Lambda_1$ likewise be the parts of
 $A+V$ associated with $\Ran Q$ and $\Ran Q^\perp$.

 Assume that there is $d>0$ such that
 \begin{equation}\label{eq:symmSinThetaSpecSep}
  \dist\bigl(\spec(A_0),\spec(\Lambda_1)\bigr) \ge d \quad\text{ and }\quad \dist\bigl(\spec(A_1),\spec(\Lambda_0)\bigr) \ge d\,.
 \end{equation}
 If $V\in\fS$ for some symmetrically-normed ideal $\fS$ with norm $\tnorm{\cdot}$ such that $(\fS,\tnorm{\cdot})$ admits Ky Fan's
 dominance theorem, then one has $\sin\Theta=\abs{P-Q}\in\fS$ and
 \[
  \tnorm{\sin\Theta} = \tnorm{P-Q}\le \frac{\pi}{2}\,\frac{\tnorm{V}}{d}\,,
 \]
 where $\Theta=\Theta(P,Q)$ is the operator angle associated with $\Ran P$ and $\Ran Q$.
 \begin{proof}
  Denote
  \[
   X:=P-Q=PQ^\perp-P^\perp Q \quad\text{ and }\quad T:=PVQ^\perp-P^\perp VQ\,.
  \]

  We show that the operator $X$ satisfies
  \begin{equation}\label{eq:SylSolRepr}
   \langle y, Xx \rangle = \int_\R \langle y, \ee^{\ii t A}T \ee^{-\ii t(A+V)}x \rangle f_d(t)\,\dd t \quad\text{ for }\quad
   x,y \in\cH\,,
  \end{equation}
  where $f_d\in L^1(\R)$ is any function as in Theorem \ref{thm:Sylvester}.

  Since $\Ran P$ is reducing for $A$, the projection $P$ commutes with $A$, that is, $Px\in\Dom(A)$ and $PAx=APx$ for all
  $x\in\Dom(A)=\Dom(A+V)$, see, e.g., \cite[Section III.5.6]{Kato66}. Analogously, $Q^\perp$ commutes with $A+V$. Hence, one has
  \[
   \Ran\bigl(PQ^\perp|_{\Dom(A)}\bigr)\subset \Dom(A)\cap\Ran P
  \]
  and
  \[
   PQ^\perp(A+V)x - APQ^\perp x = P(A+V)Q^\perp x - PAQ^\perp x = PVQ^\perp x
  \]
  for all $x\in\Dom(A)$, that is, the operator $Y:=PQ^\perp$ is a strong solution to the operator Sylvester equation
  $Y(A+V)-AY=PVQ^\perp$; cf.\ \cite[Section 2]{BDM83} and \cite[Section 2]{McEachin93}, and also the proof of
  \cite[Proposition 3.4]{AM12:2}.

  Since by \eqref{eq:symmSinThetaSpecSep} the spectra of the parts $A_0$ and $\Lambda_1$ are separated with distance at least $d$,
  it follows from Theorem \ref{thm:Sylvester} that
  \begin{equation}\label{eq:SylSolWeakPQp}
   \langle y,PQ^\perp x \rangle = \int_\R \bigl\langle y,\ee^{\ii t A}PVQ^\perp \ee^{-\ii t(A+V)}x \bigr\rangle f_d(t)\,\dd t
   \quad\text{ for }\quad x,y\in\cH\,.
  \end{equation}
  Indeed, by functional calculus the subspaces $\Ran P$ and $\Ran Q^\perp$ are reducing for $\ee^{\ii tA}$ and $\ee^{-\ii t(A+V)}$,
  respectively, and the associated parts are given by $\ee^{\ii tA_0}$ and $\ee^{-\ii t\Lambda_1}$. Hence, for $x\in\Ran Q^\perp$
  and $y\in\Ran P$ equation \eqref{eq:SylSolWeakPQp} agrees with \eqref{eq:SylSol}, and for $x\in\Ran Q$ or $y\in\Ran P^\perp$
  equation \eqref{eq:SylSolWeakPQp} holds since both sides of \eqref{eq:SylSolWeakPQp} are trivial in this case.

  Since the spectra of $A_1$ and $\Lambda_0$ are likewise separated with distance at least $d$, the analogous reasoning shows that
  \begin{equation}\label{eq:SylSolWeakPpQ}
   \langle y,P^\perp Qx \rangle = \int_\R \bigl\langle y,\ee^{\ii t A}P^\perp VQ \ee^{-\ii t(A+V)}x \bigr\rangle f_d(t)\,\dd t
   \quad\text{ for }\quad x,y\in\cH\,.
  \end{equation}
  Combining \eqref{eq:SylSolWeakPQp} and \eqref{eq:SylSolWeakPpQ} yields \eqref{eq:SylSolRepr}.

  Taking into account representation \eqref{eq:KyFanSup} and the fact that every Ky Fan norm is unitary invariant, it follows from
  \eqref{eq:SylSolRepr} that for arbitrary orthonormal systems $\{x_j\}_{j=1}^n$ and $\{y_j\}_{j=1}^n$ in $\cH$, $n\in\N$, one has
  \[
   \begin{aligned}
    \biggl|\sum_{j=1}^n\langle y_j,Xx_j\rangle\biggr|
    &\le \int_\R \biggl|\sum_{j=1}^n\langle y_j,\ee^{\ii tA}T\ee^{-\ii t(A+V)}x_j\rangle\biggr|\,\abs{f_d(t)}\,\dd t\\
    &\le \int_\R \tnorm{\ee^{\ii tA}T\ee^{-\ii t(A+V)}}_n\,\abs{f_d(t)}\,\dd t = \norm{f_d}_{L^1(\R)}\cdot\tnorm{T}_n\,.
   \end{aligned}
  \]
  Again by \eqref{eq:KyFanSup}, this implies that $\tnorm{X}_n\le\norm{f_d}_{L^1(\R)}\cdot\tnorm{T}_n$ and hence, in view of the
  identity \eqref{eq:Nagy}, that $\tnorm{X}_n\le\frac{\pi}{2d}\cdot\tnorm{T}_n$ for all $n\in\N$. Since $(\fS,\tnorm{\cdot})$
  admits Ky Fan's dominance theorem, one concludes that $X=P-Q\in\fS$ and
  \[
   \tnorm{P-Q} \le \frac{\pi}{2}\,\frac{\tnorm{T}}{d}\,.
  \]

  Due to the fact that $\sin\Theta=\abs{P-Q}\in\fS$ and $\tnorm{\sin\Theta}=\tnorm{P-Q}$ by \eqref{eq:sinTheta} and
  \eqref{eq:snNormPolar}, it remains to show that
  \[
   \tnorm{T} \le \tnorm{V}\,.
  \]
  Indeed, one observes that (cf.\ \cite[Lemma 6.2]{DK70})
  \[
   (P-P^\perp)V - V(Q-Q^\perp) = 2PVQ^\perp - 2P^\perp VQ = 2T\,,
  \]
  and, therefore, $2\tnorm{T} \le 2\tnorm{V}$ since $P-P^\perp$ and $Q-Q^\perp$ are unitary.
 \end{proof}%
\end{proposition}

In a similar way, one can also prove a corresponding extension of Proposition \ref{prop:sinTheta0}, but we do not need this here.

%%%%%%%%%%%%%%%%%%%%%%%%%%%%%%%%%%%%%%%%%%%%%%%%%%%%%%%%%%%%%%%%%%%%%%%%%%%%%%%%%%%%%%%%%%%%%%%%%%%%%%%%%%%%%%%%%%%%%%%%%%%%%%%%%%%
%%% sin2\Theta theorem for symmetrically-normed ideals
%%%%%%%%%%%%%%%%%%%%%%%%%%%%%%%%%%%%%%%%%%%%%%%%%%%%%%%%%%%%%%%%%%%%%%%%%%%%%%%%%%%%%%%%%%%%%%%%%%%%%%%%%%%%%%%%%%%%%%%%%%%%%%%%%%%
Using Proposition \ref{prop:symmSinThetasnIdeals} instead of Proposition \ref{prop:symmSinTheta} in the proof of Theorem
\ref{thm:intro}, we obtain the following extension of the $\sin2\Theta$ theorem to symmetrically-normed ideals.

\begin{theorem}[The generic $\sin2\Theta$ theorem for symmetrically-normed ideals]\label{thm:sin2ThetasnIdeals}
 Let $A$, $V$, and $Q$ be as in Theorem \ref{thm:intro}. If $V\in\fS$ for some symmetrically-normed ideal $\fS\subset\cL(\cH)$ with
 norm $\tnorm{\cdot}$ such that $(\fS,\tnorm{\cdot})$ admits Ky Fan's dominance theorem, then the operator angle
 $\Theta=\Theta(\EE_A(\sigma),Q)$ satisfies $\sin2\Theta\in\fS$ and
 \[
  \tnorm{\sin2\Theta} \le \frac{\pi}{2}\cdot 2\,\frac{\tnorm{V}}{d}\,.
 \]
\end{theorem}

%%%%%%%%%%%%%%%%%%%%%%%%%%%%%%%%%%%%%%%%%%%%%%%%%%%%%%%%%%%%%%%%%%%%%%%%%%%%%%%%%%%%%%%%%%%%%%%%%%%%%%%%%%%%%%%%%%%%%%%%%%%%%%%%%%%
%%% Concluding remark
%%%%%%%%%%%%%%%%%%%%%%%%%%%%%%%%%%%%%%%%%%%%%%%%%%%%%%%%%%%%%%%%%%%%%%%%%%%%%%%%%%%%%%%%%%%%%%%%%%%%%%%%%%%%%%%%%%%%%%%%%%%%%%%%%%%
We close this work with a concluding observation that addresses the case of normal operators $A$.

\begin{remark}\label{rem:sin2ThetasnIdeals}
 A statement analogous to Proposition \ref{prop:symmSinThetasnIdeals} holds if the operator $A$ is assumed to be only normal and
 $V$ is just bounded such that $A+V$ is normal. In this situation, the constant $\frac{\pi}{2}$ in the resulting estimate has to be
 replaced by some suitable constant less than $2{.}91$, see Remark \ref{rem:Sylvester}. Consequently, Theorem
 \ref{thm:sin2ThetasnIdeals} holds with the same modification if $A$ is just normal and $V$ is only bounded. Indeed, in this case,
 the operator $D$ in \eqref{eq:defD} is normal as well.

 Note that either of these constants is universal in the sense that it does not depend on $A$, $V$, or $(\fS,\tnorm{\cdot})$.
 However, in some special cases a better constant is available: If, for example, the involved pairs of spectral sets satisfy one of
 the additional assumptions mentioned in Remark \ref{rem:Sylvester}, or if one restricts the considerations to the ideal of
 Hilbert-Schmidt operators with its usual norm, then the constant in the corresponding estimates can be replaced by $1$,
 cf.\ \cite[Theorem 6.1]{BDM83}. As in Remark \ref{rem:Sylvester}, other improvements on the constant may be available for small
 dimensions of the underlying Hilbert space.
\end{remark}

%%%%%%%%%%%%%%%%%%%%%%%%%%%%%%%%%%%%%%%%%%%%%%%%%%%%%%%%%%%%%%%%%%%%%%%%%%%%%%%%%%%%%%%%%%%%%%%%%%%%%%%%%%%%%%%%%%%%%%%%%%%%%%%%%%%
%%% Acknowledgements
%%%%%%%%%%%%%%%%%%%%%%%%%%%%%%%%%%%%%%%%%%%%%%%%%%%%%%%%%%%%%%%%%%%%%%%%%%%%%%%%%%%%%%%%%%%%%%%%%%%%%%%%%%%%%%%%%%%%%%%%%%%%%%%%%%%
\section*{Acknowledgements.}
The author is indebted to his Ph.D. advisor Vadim Kostrykin for introducing him to this field of research and fruitful discussions.
He would also like to express his gratitude to Konstantin Makarov and Julian Gro{\ss}mann for useful remarks.

%%%%%%%%%%%%%%%%%%%%%%%%%%%%%%%%%%%%%%%%%%%%%%%%%%%%%%%%%%%%%%%%%%%%%%%%%%%%%%%%%%%%%%%%%%%%%%%%%%%%%%%%%%%%%%%%%%%%%%%%%%%%%%%%%%%
%%%%%%%%%%%%%%%%%%%%%%%%%%%%%%%%%%%%%%%%%%%%%%%%%%%%%%%%%%%%%%%%%%%%%%%%%%%%%%%%%%%%%%%%%%%%%%%%%%%%%%%%%%%%%%%%%%%%%%%%%%%%%%%%%%%
%%% Bibliography
%%%%%%%%%%%%%%%%%%%%%%%%%%%%%%%%%%%%%%%%%%%%%%%%%%%%%%%%%%%%%%%%%%%%%%%%%%%%%%%%%%%%%%%%%%%%%%%%%%%%%%%%%%%%%%%%%%%%%%%%%%%%%%%%%%%
%%%%%%%%%%%%%%%%%%%%%%%%%%%%%%%%%%%%%%%%%%%%%%%%%%%%%%%%%%%%%%%%%%%%%%%%%%%%%%%%%%%%%%%%%%%%%%%%%%%%%%%%%%%%%%%%%%%%%%%%%%%%%%%%%%%

\end{document}